\documentclass[12pt]{amsart}
\usepackage{amssymb}
\usepackage{multicol}
\usepackage{colordvi}
\usepackage{graphicx}

\numberwithin{equation}{section}
\newtheorem{theorem}{Theorem}[section]
\newtheorem{proposition}[theorem]{Proposition}
\newtheorem{lemma}[theorem]{Lemma}

\theoremstyle{definition}
\newtheorem{definition}[theorem]{Definition}
\newtheorem{remark}[theorem]{Remark}

\newcounter{FNC}[page]
\def\fauxfootnote#1{{\addtocounter{FNC}{2}\Magenta{$^\fnsymbol{FNC}$}%
     \let\thefootnote\relax\footnotetext{\Magenta{$^\fnsymbol{FNC}$#1}}}}

\newcommand{\dcol}[1]{\Blue{#1}}
\newcommand{\demph}[1]{\dcol{{\sl#1}}}

\newcommand{\rk}{\mbox{\rm rank}}
\newcommand{\srk}{\mbox{\scriptsize\rm rank}}

\newcommand{\C}{\mathbb{C}}

\newcommand{\R}{\mathbb{R}}
\newcommand{\Z}{\mathbb{Z}}
\renewcommand{\P}{\mathbb{P}}

\newcommand{\bO}{{\mathbb O}}
\newcommand{\bS}{{\mathbb S}}
\newcommand{\bT}{{\mathbb T}}
 
\newcommand{\calA}{{\mathcal A}}
\newcommand{\calO}{{\mathcal O}}
\newcommand{\calS}{{\mathcal S}}

\title[Lower bounds and orientability of real toric varieties]{Lower bounds in 
       real algebraic geometry and orientability of real toric varieties}

\author{Evgenia Soprunova}
\address{Department of Mathematics\\
        Kent State University\\
        Summit Street, Kent, OH 44242\\
        USA}
\email{soprunova@math.kent.edu}
\urladdr{http://www.math.kent.edu/~soprunova/}
\author{Frank Sottile}
\address{Frank Sottile\\
         Department of Mathematics\\
         Texas A\&M University\\
         College Station\\
         Texas \ 77843\\
         USA}
\email{sottile@math.tamu.edu}
\urladdr{www.math.tamu.edu/\~{}sottile}
\thanks{Research of Sottile supported in part by NSF grant DMS-1001615, and the
  Mittag-Leffler Institut}
\subjclass{14M25, 57B20, 57S10, 14P99}
\begin{document}

\begin{abstract}
 The real solutions to a system of sparse polynomial equations may be realized as a fiber
 of a projection map from a toric variety.
 When the toric variety is orientable, the degree of this map is a lower bound for the
 number of real solutions to the system of equations.
 We strengthen previous work by characterizing when the toric variety is
 orientable. 
 This is based on work of Nakayama and Nishimura, who characterized the
 orientability of smooth real toric varieties.
\end{abstract}

\maketitle

\section*{Introduction}

A ubiquitous phenomenon in enumerative real algebraic geometry 
is that many geometric problems possess a 
non-trivial lower bound on their number of real solutions.
For example, at least 3 of the 27 lines on a real cubic surface are real as are at least 8 of
the  12 rational cubics interpolating 8 real points in the 
plane~\cite[Prop.\ 4.7.3]{DeKh00}, but there are many, many other 
examples~\cite{AG11, EG02, FKh, secant, OT11, OT12, RSSS, SS06, W}.
This phenomenon has the potential for significant
impact on the applications of mathematics as a nontrivial lower bound is an existence proof
of real solutions.
For this potential to be realized, methods need to be developed to predict when a
system of polynomial equations or a geometric problem has a lower bound on its number of
real solutions and to compute this bound.

We developed a theory of lower bounds on the number of real solutions to systems of sparse
polynomials~\cite{SS06}.
There, a system of polynomial equations was
formulated as a 
fiber of a projection map from a toric subvariety of a sphere.
When the 
toric variety is
orientable, the absolute value of the
degree of this projection map is a lower bound on the number of
real solutions.
Besides giving a condition implying this orientability, a method (foldable triangulations of
the Newton polytope) was developed to compute the degree of certain maps, and a
class of examples of polynomial systems (Wronski polynomial systems from posets) was presented 
to which this theory applied.

Further work~\cite{JW07,JZ12} on foldable triangulations 
has advanced our understanding of the bound they give.
Others~\cite{AG11,FKh,modfour,OT11, OT12} have developed additional methods for proving lower bounds
in real algebraic geometry and experimentation~\cite{RSSS, SS06, HS} has revealed many more
likely examples of lower bounds. 

We characterize 
which sparse polynomial systems possess a lower bound in the context of~\cite{SS06},
by extending
work of Nakayama and Nishimura~\cite{NN05}, who characterized
the orientability of small covers, which are
topological versions of smooth real projective toric varieties.
We characterize the orientability of the smooth points of any 
real toric variety, as well as toric subvarieties of a sphere,
solving an important open problem from~\cite{SS06}.

We review the construction of real toric varieties and spherical toric varieties in
Section~1, where we 
formulate our results on orientability.
Section~2 contain the mildly technical proof of these results.
In Section~3 we use this characterization of orientability
to strengthen results from~\cite{SS06} on the theory of lower bounds for the number of
real solutions to systems of sparse polynomials.

%
\section{Constructions of real toric varieties}

Real toric varieties appear in many applications of mathematics~\cite{TDS,Kr02,ASCB}
and are interesting objects in their own right~\cite{De04}.
Davis and Januszkiewicz~\cite{DJ91} introduced the notion of a small cover of a simple
convex polytope as a generalization of smooth projective real toric varieties.
We 
describe 
real toric varieties and small covers in terms of the gluing of
explicit cell complexes and give a mild extension of  Davis and Januszkiewicz's notion of a
small cover (which are manifolds) to not necessarily smooth spaces. 
A projective toric variety 
may be lifted to the sphere over 
real projective space, and we also describe these spherical toric varieties
in terms of the gluing of explicit cell complexes.

Nakayama and Nishimura~\cite{NN05} used this presentation of small covers to characterize 
their orientability, and similar arguments characterize
the orientability of the smooth points of the above spaces.

Real toric varieties, singular small covers, and toric subvarieties of the sphere are
obtained by gluing the real torus $\dcol{\bT^n}:=(\R^\times)^n$ or  
$\{\pm1\}\times \bT^n=\bT^{n+1}/\R_{\geq}$ along copies of $\bT^{n-1}$, one copy for each
vector in a set of integer vectors. 
There are further gluings 
in higher codimension, which presents
these spaces as explicit cell complexes.
They are smooth at the points of their dense torus $\bT^n$ (or at $\{\pm1\}\times\bT^n$)
and the attached tori $\bT^{n-1}$, and so their orientability is determined by the gluing 
along the tori $\bT^{n-1}$.

Complex toric varieties are normal varieties over $\C$ equipped with an action of an algebraic
torus $(\C^\times)^n$ having a dense orbit.
They are classified by rational fans $\Sigma$ in $\R^n$, which encode their construction
as a union of affine toric varieties $U_\sigma$, one for each cone $\sigma\in\Sigma$.
A toric variety is
a  union of disjoint torus orbits $\calO_\sigma$, one
for each cone $\sigma\in\Sigma$, with 
$\dim \calO_\sigma=n{-}\dim(\sigma)$.
The dense orbit $\calO_0$ coincides with the smallest affine patch $U_0$, and both are
associated to 
origin 0 in the fan.
See~\cite{Fu93} for a complete description.

A toric variety has a canonical set $Y$ of real points 
obtained from the real points
of the orbits $\calO_\sigma$ of the construction~\cite[Ch.~4]{Fu93}.
The dense orbit $\calO_0(\R)\simeq\bT^n$ is isomorphic to 
$(\R^\times)^n=(\R_{>0})^n\times \{\pm 1\}^n$, which has $2^n$ components, each a 
topological $n$-ball.
The subgroup $\{\pm1\}^n\subset\bT^n$ acts on 
$Y$, permuting the components of $\calO_0(\R)$.
The orbit space of $Y$ under the group $\{\pm1\}^n$ is isomorphic to the closure $Y_{\geq}$
of any component of $\calO_0(\R)$ in the usual topology (not Zariski!) on $Y$.
Each orbit $\calO_\sigma(\R)$ has a unique component 
contained in $Y_{\geq}$. 
We call this component $\dcol{F_\sigma}$ a \demph{face} of $Y_{\geq}$, which is isomorphic
to $(\R_{>0})^{n-\dim(\sigma)}$.
This endows 
$Y_{\geq}$ with the structure of a cell complex that is dual to the fan
$\Sigma$.
That is, the intersection $\overline{F_\sigma}\cap\overline{F_\tau}$ of the closures of
two faces is nonempty only if $\sigma$ and $\tau$ lie in some cone of $\Sigma$, in which
case it is the closure $\overline{F_\rho}$ where $\rho$ is the minimal such cone.

The integer points in a cone $\sigma$ of $\Sigma$ form a subsemigroup of $\Z^n$ whose image
in $(\Z/2\Z)^n=\{\pm1\}^n$ is a subgroup $\dcol{\overline{\sigma}}$ of $\{\pm1\}^n$.
This subgroup $\overline{\sigma}$ is the isotropy subgroup of the face $F_\sigma$ of
$Y_{\geq}$. 
We will write $(-1)^v=((-1)^{v_1},\dotsc,(-1)^{v_n})$ for the image of $v\in\Z^n$ in
$(\Z/2\Z)^n=\{\pm1\}^n$. 
This gives the following description of $Y$ as a quotient space of
$Y_{\geq}\times\{\pm1\}^n$. 

\begin{proposition}
 The real toric variety $Y$ is obtained as the quotient of the cell complex
 $Y_{\geq}\times\{\pm1\}^n$ by the equivalence relation where
\[
   (p,\xi)\sim(q,\eta)\ \Longleftrightarrow\ 
    p=q\ \mbox{ and }\ \xi\overline{\sigma}=\eta\overline{\sigma}\,,\ 
    \mbox{where $p$ lies in the face $F_\sigma$.}
\]
\end{proposition}

A \demph{facet} of $Y_{\geq}$ is a face $F_\sigma$ corresponding to a one-dimensional
cone $\sigma$.
The real toric variety $Y$ is smooth at points corresponding to facets, but may not be smooth
along lower-dimensional faces.
If \dcol{$Y^\circ_{\geq}$} is
the union of the dense face $F_0$ and its facets, then 
\[
   \dcol{Y^\circ}\ :=\ 
    (Y^\circ_{\geq}\times\{\pm1\}^n)/\sim
\]
consists of smooth points of $Y$.

We generalize this construction.
Let $P$ be a finite ranked poset with minimal element $0$ and rank at most $n$ where two
elements $\sigma,\tau\in P$ have at most one minimal 
upper bound in $P$. 
The cones $\sigma$ in a rational fan in $\R^n$ form such a poset.
Suppose further 
that we have a collection $\dcol{\calS}:=\{\overline{\sigma}\mid\sigma\in P\}$
of subgroups of $\{\pm1\}^n$ where $\overline{\sigma}\simeq\{\pm1\}^{\srk(\sigma)}$, and if
$\sigma\subset\tau$, then $\overline{\sigma}\subset\overline{\tau}$.
Finally,
suppose that we have a cell complex $\Delta$ with cells (called faces) indexed by
elements of $P$,
\[
   \dcol{\Delta}\ =\ \coprod_{\sigma\in P} F_\sigma,
\]
where each face $F_\sigma$ is a cell of dimension $n-\rk(\sigma)$, which 
we identify with the interior of the closed unit ball in $\R^{n-\srk(\sigma)}$.
We further suppose that:
\begin{itemize}
 \item  $\Delta$ is a subset of the closed ball $\overline{F_0}$ in $\R^n$, 
 \item  the closure of a face $F_\sigma$ in $\R^n$ is homeomorphic to the closed ball of
   dimension $n-\rk(\sigma)$, and 
 \item  given $\sigma, \tau\in P$, the closures of the faces $F_\sigma$ and $F_\tau$
   either do not meet (if $\sigma$ and $\tau$ have no upper bound in $P$), or their
   intersection is the closure of the face $F_\rho$, where $\rho$ is the least upper bound
   of $\sigma$ and $\tau$ in $P$.
\end{itemize}

\begin{definition}
  Given a ranked poset $P$, system $\calS$ of subgroups of
  $\{\pm1\}^n$, and a cell complex $\Delta$ as above, the \demph{small cover} 
  \dcol{$Y(\Delta,\calS)$} of $\Delta$
  is the quotient 
\[
   (\Delta\times\{\pm1\}^n)/\sim\,,
\]
  where $(p,\xi)\sim(q,\eta)$ if and only if $p=q$ and
  $\xi\overline{\sigma}=\eta\overline{\sigma}$,
  where $p$ lies in the face $F_\sigma$.

  Observe that $Y(\Delta,\calS)$ is equipped with a natural action of $\{\pm1\}^n$ whose
  orbit space is $\Delta$, where the orbit of a face $F_\sigma$ is identified
  with $F_\sigma\times \{\pm1\}^n/\overline{\sigma}\simeq \bT^{n-\srk(\sigma)}$.
  In particular, it is a $\{\pm1\}^n$-equivariant compactification of $\bT^n$.
\end{definition}

A real toric variety $Y$ associated to a fan $\Sigma$ is a small cover where $P$ is the set of 
cones in the fan, $\Delta=Y_\geq$, and $\calS=\{\overline{\sigma}\mid \sigma\in \Sigma\}$.

The points of $Y(\Delta,\calS)$ corresponding to the big cell $F_0$ and to 
facets $F_\sigma$ are points where $Y(\Delta,\calS)$ is a topological manifold.
Write $\Delta^\circ$ for the union of the big cell and the facets, and 
$Y^\circ(\Delta,\calS)=(\Delta^\circ\times\{\pm1\}^n)/\sim$ for this subset of the smooth
points of  $Y(\Delta,\calS)$.

Let $\Delta\subset\R^n$ be a 
$n$-dimensional polytope with integer vertices and 
normal fan $\Sigma$.
Then the real toric variety $Y_\Sigma$ associated to $\Sigma$ has a projective embedding given
by $\Delta$. 
We may assume that the integer points $\Delta\cap\Z^n$ generate $\Z^n$.
Let $\P^\Delta$ be the real projective space with coordinates
indexed by $\Delta\cap\Z^n$ and $\dcol{y^\alpha}:=y_1^{\alpha_1}\dotsb y_n^{\alpha_n}$
the monomial with exponent $\alpha$.
Then we have an injection
 \begin{equation}\label{Eq:varphi}
  \varphi_\Delta\ \colon\  \bT^n\ \ni\ y\ \longmapsto\ 
   [y^\alpha\mid \alpha\in\Delta\cap\Z^n]\,,
 \end{equation}
where $[\dotsb]$ denotes homogeneous coordinates for $\P^\Delta$.
The closure $Y_\Delta$ of the image of this map is isomorphic to the real toric
variety $Y_\Sigma$, and the cell complex $Y^\circ_\geq$ is identified with the polytope
$\Delta$. 

The unit sphere $\bS^\Delta\subset\R^\Delta$ has a two-to-one map to 
$\P^\Delta$, and we define $Y^+_\Delta$ to be the pullback of $Y_\Delta$ along this map.
The sphere $\bS^\Delta$ has homogeneous coordinates
$(x_\alpha\mid\alpha\in\Delta\cap\Z^n)$, where we identify points 
with a positive constant of proportionality.
The group $\{\pm1\}^{n+1}$ acts on $\bS^\Delta$ with the last coordinate acting through
global multiplication by $\pm1$ and the remaining
coordinates $\{\pm1\}^n$ 
through the map $\varphi_\Delta$~\eqref{Eq:varphi},
\[
   (g,g_{n+1}).(x_\alpha\mid\alpha\in\Delta\cap\Z^n)\ =\ 
   (g_{n+1} g^\alpha x_\alpha\mid\alpha\in\Delta\cap\Z^n)\,.
\]

The faces of $Y^+_\Delta$ are its intersections with coordinate subspaces $\bS^F$ of
$\bS^\Delta$ corresponding to faces $F$ of $\Delta$,
\[
  \bS^F\ :=\ \{(x_\alpha\mid\alpha\in\Delta\cap\Z^n)
      \mid x_\alpha=0\mbox{ if }\alpha\not\in F\cap\Z^n\}\,.
\]
The isotropy subgroup of $\bS^F$ is
\[
   \{(g,g_{n+1})\mid g_{n+1}g^\alpha=1\mbox{ for }\alpha\in F\cap\Z^n\}\,.
\]

Vectors $b$ in the normal cone $\sigma_F$ to a face $F$ of $\Delta$ have constant dot
product with elements of $F$---define $b\cdot F$ to be this constant.
Then the subgroup 
\[
   {\overline{\sigma}_F}^+\ :=\ 
   \{ (-1)^{(b,b\cdot F)}\mid b\in\sigma_F\}
   \ \subset\ \{\pm1\}^{n+1}
\]
is the isotropy group of $\bS^F$, and therefore of the corresponding face of
$Y^+_\Delta$.

\begin{proposition}
  The spherical toric variety $Y^+_\Delta$ is obtained as the quotient of the cell complex
  $\Delta\times\{\pm1\}^{n+1}$ by the equivalence relation
\[
   (p,\xi)\sim(q,\eta)\ \Longleftrightarrow\ 
    p=q\ \mbox{ and }\ \xi{\overline{\sigma}_F}^+=\eta{\overline{\sigma}_F}^+\,,\ 
    \mbox{where $p$ lies in the face $F$.}
\]
\end{proposition}

%
\section{Characterization of orientability}\label{S:two}

We follow Nakayama and Nishimura~\cite{NN05} to characterize the orientability of a 
general small cover and of spherical toric varieties, and determine their numbers of components.

\begin{theorem}\label{Th:small_cover}
   Let $Y(\Delta, S)$ be a small cover of dimension $n$.
\begin{itemize}
  \item[(1)]  
     $Y^{\circ}(\Delta, S)$ is orientable  if and only if there exists a basis of 
     $\{\pm 1\}^n$ such that for every $\sigma\in P$ of rank $1$, 
     the generator of\/ $\overline{\sigma}\simeq\{\pm1\}$ is a product of an odd number of
     basis vectors. 
\item[(2)] 
     The components of\/ $Y^{\circ}(\Delta, S)$ are naturally indexed by 
     $\{\pm1\}^n/\langle \overline{\sigma}\mid \rk(\sigma) =1\rangle$. 
\end{itemize}
\end{theorem}

 Thus $Y^{\circ}(\Delta, S)$ has $2^{n-k}$ connected components, where 
$2^k=|\langle \overline{\sigma}\mid \rk(\sigma) =1\rangle|$.

\begin{proof}
 For each $\sigma\in P$ with rank 1, let $g_\sigma$ be the generator of
 $\overline{\sigma}\simeq\Z/2\Z$.  
 Then $Y^\circ:=Y^{\circ}(\Delta, S)$ is obtained by gluing $(\Delta,\xi)$ and
 $(\Delta,\eta)$ along $F_\sigma$ whenever $\xi=\eta g_\sigma$ for some $\sigma\in P$ of
 rank 1, so the  connected components of $Y^\circ$  correspond to the 
 orbits of $Y^\circ$ under the action of 
 $\langle \overline{\sigma}\mid \rk(\sigma) = 1\rangle$. 

 The space $Y^{\circ}$ is orientable if and only if 
 $H_n(Y^{\circ},\Z)\neq \{0\}$.
 This group is the kernel $\ker\partial$ of the differential in the 
 cellular chain complex of the cell complex $Y^{\circ}$,
\[
   C_n\ \xrightarrow{\ \partial\ }\ C_{n-1}\,.
\]
 Here $C_n$ is the free abelian group generated by
\[
  \{\Delta\}\times\{\pm 1\}^n\ =\ \left\{(\Delta,\xi)\mid \xi\in\{\pm 1\}^n\right\}
\]
  and $C_{n-1}$ is the free abelian group generated by 
\[
   \left\{[F_\sigma,\xi]\mid \sigma\in P, \rk(\sigma) =1 , \xi\in \{\pm 1\}^n\right\}/\sim\,,
\]
 where $[F_{\sigma},\xi]\sim [F_\sigma,\xi g_\sigma]$. 
 Orient each facet $F_{\sigma}$ so that 
\[
   \partial(\Delta)\ =\ \sum_{\srk(\sigma)=1} F_{\sigma}\ .
\]
Consider an $n$-cycle 
\[
  X\ =\ \sum_{\xi\in \{\pm 1\}^n} n_\xi \cdot (\Delta,\xi)\ \in\ C_n
\]
on $Y^{\circ}$, where $n_\xi\in\Z$.
Then 
\[
  \partial(X)\ =\ 
   \sum_{\xi\in \{\pm 1\}^n}  n_\xi \sum_{\srk(\sigma)=1}[F_{\sigma},\xi]
   \ =\ \sum_{\srk(\sigma)=1}\ \sum_{\xi\in\{\pm1\}^n/\langle g_\sigma\rangle}
    (n_\xi+n_{\xi g_\sigma})[F_\sigma,\xi]\,.
\]
 Hence an $n$-cycle $X$ lies in $\ker\partial$ if and only if 
 $n_\xi=-n_{\xi g_\sigma}$ for all $\xi$ in $\{\pm 1\}^n$ and $\sigma$ of rank~1.  
 Equivalently, $n_\xi =(-1)^k n_{\xi g_{\sigma_1}\cdots\ g_{\sigma_k}}$ for
 all $\xi\in\{\pm 1\}^n$ and $\sigma_i$ of rank 1. 

 We show that $\ker\partial$ is non-trivial if and and only if there exists a
 basis  $e_1,\dots, e_n$ of $\{\pm 1\}^n$ such 
 $g_\sigma$ is a product of an odd number of basis
 vectors, for each element $\sigma\in P$ of rank one. 
 Let $\bO$ be the set of generators $g_\sigma$ of $\overline{\sigma}$ for rank one elements
 $\sigma\in P$.

 Suppose that there exists a basis  $e_1,\dots, e_n$ of $\{\pm 1\}^n$ such that each 
 $g_\sigma\in\bO$ is a product of an odd number of basis vectors.  
 For $\xi\in\{\pm 1\}^n$ define $n_\xi$ to be 1 if $\xi$ is a product of an even number of
 the $e_i$ and $-1$ if it is a product of  an odd number of the $e_i$. 
 Then 
 $n_\xi=-n_{\xi g_{\sigma}}$  for all $\xi$ and $\sigma$, so 
 $\ker\partial$ is non-trivial and hence $Y^{\circ}$ is orientable. 
 Since the number of connected components is $2^{n-k},$ the kernel is isomorphic to
 $\Z^{2^{n-k}}$.  

 If there is no such basis of $\{\pm1\}^n$, then there is some $g_\sigma\in\bO$ 
 which is a product of an even number of other elements in $\bO$, for 
 otherwise we can reduce $\bO$ to a linearly independent set and then extend it
 to a basis of $\{\pm 1\}^n$.  
 We get $g_\sigma=g_{\sigma_1}\cdots g_{\sigma_{2k}}$ and hence $1=g_\sigma g_{\sigma_1}\cdots
 g_{\sigma_{2k}}$, so  for every $\xi$ we get
\[
   n_{\xi}\ =\ (-1)^{2k+1}n_{\xi g_\sigma g_{\sigma_1}\cdots g_{\sigma_{2k}}}
          \ =\ -n_\xi\,,
\]
 which implies that $n_\xi=0$ and hence $\ker\partial =0$ and so $Y^{\circ}$ is
 non-orientable.  
\end{proof}

We restate the orientability criteria of Theorem~\ref{Th:small_cover} for real toric
varieties.

\begin{theorem}\label{T:orient_toric}
 Let $Y$ be a real toric variety defined by a fan $\Sigma$.
 Then $Y^{\circ}$ is orientable if and only if there exists a basis of 
 $\{\pm 1\}^n$ such that $(-1)^b$ is  a product of an odd
 number of basis vectors, for each primitive vector $b$ lying on a ray of $\Sigma$. 
\end{theorem}

The condition of Theorem~\ref{T:orient_toric} is easily checked.

\begin{lemma}\label{L:cond} 
Given $A\subset\{\pm 1\}^n,$ the condition that there exists a basis
of  $\{\pm 1\}^n$ such that each vector in $A$ is a product of an odd number
of basis vectors, is equivalent to the condition that no product of an odd number of
vectors in $A$ is equal to $1$ in $\{\pm 1\}^n$. 
\end{lemma}

\begin{proof}
If we had $v_1\cdots v_{2k+1}=1,$ then $v_{2k+1}=v_1\cdots v_{2k},$ and expressing each
$v_i$ as the product of an odd number of basis elements of  $\{\pm 1\}^n$ yields a
contradiction. 
For the other implication,  reduce $A$ to a linearly independent set $A'$ and then extend
$A'$ to a basis of $\{\pm 1\}^n$.  
If there were a vector in $A\setminus A'$ which is a product of an even number of vectors
$v=v_1\cdots v_{2k},$ we would have then had $v\cdot v_1\cdots v_{2k}=1$. 
\end{proof}

We may check if the condition is satisfied by reducing $A$ to a
linearly independent set $A'$ and checking if each vector in $A\setminus A'$ is a product
of an odd number vectors in $A'$. 

The analog of Theorem~\ref{Th:small_cover} for spherical toric varieties
has a similar proof. 

\begin{theorem}\label{Th:spherical}
 Let $Y^+_\Delta\subset\bS^\Delta$ be a spherical toric variety defined by a full-dimensional
 lattice polytope $\Delta\subset\R^n$. 
\begin{itemize}
  \item[(1)]  
     $Y^+_\Delta$ is orientable if and only if there exists a basis of\/ 
     $\{\pm 1\}^{n+1}$ such that for each facet $F$ of $\Delta$ with primitive normal vector
     $b$,  the element $(-1)^{(b,b\cdot F)}$ is a product of an odd number of basis elements.
\item[(2)] 
     The components of $Y^+_\Delta$  are naturally indexed by 
   \[
      \{\pm1\}^{n+1}/\langle (-1)^{(b,b\cdot F)}\mid 
      \mbox{$b$ is a primitive normal vector to a facet $F$ of $\Delta$}\rangle\,.
   \] 
\end{itemize}
\end{theorem}

\section{Examples and applications to lower bounds}

We settle questions of orientability 
left open in~\cite{SS06} and explain our motivation 
from the study of real solutions to systems of polynomials.
We begin with an example.

\subsection{Cross Polytopes}

The cross polytope is the convex hull of the basis vectors
$e_1,\dotsc,e_n$ in $\R^n$ and their negatives $-e_1,\dotsc,-e_n$.
When $n>1$ the corresponding toric variety is singular.
The rays of its normal fan have generators 
$(\pm 1,\dotsc,\pm 1)$, all with the same image in $(\Z/2\Z)^n$.
The hypotheses of Theorem~\ref{Th:small_cover} hold, and so
the corresponding real toric variety is orientable and its smooth points have $2^{n-1}$
connected components.
Figure~\ref{F:diamond} displays the cross polytope when $n=2$ and an embedding in $\R^3$ of
the corresponding real toric variety.
\begin{figure}[htb]
   \raisebox{18pt}{\includegraphics[height=100pt]{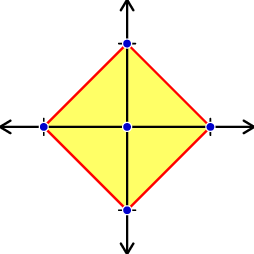}}
    \qquad\qquad
   \includegraphics[height=128pt]{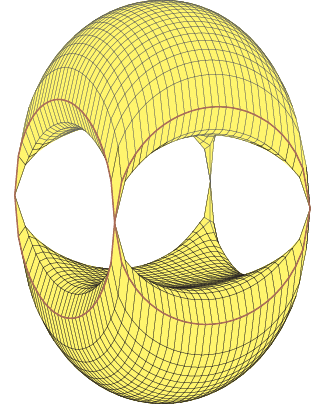}
 \caption{Two-dimensional cross polytope and double pillow.}
 \label{F:diamond}
\end{figure}
This example was treated in detail in~\cite[\S~7]{So03}.

%
\subsection{Order Polytopes}

The \demph{order polytope $O(P)$}~\cite{St86} of a finite poset $P$ is
\[
   O(P)\ :=\ \{y\in[0,1]^P\mid 
     a\leq b\mbox{ in }P\ \Rightarrow\ y_a\leq y_b\}\,.
\]
The integer points of $O(P)$ are its vertices and they correspond to the order ideals
of $P$.

\begin{theorem}
  $Y_{O(P)}$ is orientable if and only if all maximal chains of
  $P$ have odd length.  
\end{theorem}

\begin{proof}
 Lemma~4.9 of~\cite{SS06} (or rather its proof) implies that $Y_{O(P)}$ is orientable if
 all maximal chains of $P$ have odd length.
 We establish the converse.

The order polytope has three types of facets
\[
  \begin{array}{ll}
   y_a=0\quad &\textrm{for $a\in P$ minimal},\\
   y_b=1\quad &\textrm{for $b\in P$ maximal},\\
   y_b-y_a=0\quad &\textrm{for}\ b\  \textrm{covering}\ a\ (a\lessdot b)\ \textrm{in}\ P\,.
 \end{array}
\]
Replacing $=$ by $\geq$ gives valid inequalities 
for $O(P)$, which we write in matrix form
 \begin{equation}\label{Eq:OP}
   O(P)\ :=\ \left\{y\in \R^P \mid \calA y\geq c \right\}\,.
 \end{equation}
By Theorem~\ref{T:orient_toric}, $Y_{O(P)}$ is orientable 
if and only if there is a basis of the row space of $\calA$, reduced modulo 2, such
that each row is a sum of an odd number of basis vectors. 

Fix a maximal chain $a_1\lessdot\dots\lessdot a_k$ in $P$.
The corresponding facets of $O(P)$ are
\[
 y_{a_1}\ =\ 0\,,\quad y_{a_2}-y_{a_1}\ =\ 0\,,\quad\dots,\quad
 y_{a_k}-y_{a_{k-1}}\ =\ 0\,,
\quad y_{a_k}\ =\ 1\,,
\]
and the corresponding rows of the matrix $\calA$ (modulo 2) are 
\[
 \left[
 \begin{array}{rrrcrrrcc}
  1&0&0&\dots&0&0&0&\dots&0\\
  1&1&0&\dots&0&0&0&\dots&0\\
  0&1&1&\dots&0&0&0&\dots&0\\
  \vdots &\vdots &\vdots &\ddots&\vdots &\vdots&\vdots&\dots&\vdots\\
  0&0&0&\dots&1&0&0&\dots&0\\
  0&0&0&\dots&1&1&0&\dots&0\\
  0&0&0&\dots&0&1&0&\dots&0
 \end{array}
 \right]\ ,
\]
where the non-zero columns correespond to $a_1,\dotsc,a_k$.
This gives $k{+}1$ rows whose sum is zero modulo 2.
If $k$ is even, Lemma~\ref{L:cond} implies that $Y_{O(P)}$ is non-orientable. 
\end{proof}

A poset $P$ is \demph{ranked modulo 2} if all maximal chains in $P$ have
the same parity.

\begin{theorem}
 A spherical toric variety $Y_{O(P)}^+$ is orientable if and only if 
 $P$ is ranked modulo $2$.
\end{theorem}

\begin{proof}
 By Lemma~4.9 of~\cite{SS06}, $Y_{O(P)}^+$ is orientable if it is ranked modulo 2.

 Suppose that $P$ is not ranked modulo 2.
 We exhibit an odd number of rows of the augmented matrix $[\calA:c]$
 whose sum is zero modulo 2, which 
 shows that $Y_{O(P)}^+$ is not orientable,
 by Theorem~\ref{Th:spherical} and Lemma~\ref{L:cond}, as these rows have the form
 $(b,b\cdot F)$ for $b$ a primitive normal to a facet of the order polytope.

 The order polytope is defined by the facet inequalities~\eqref{Eq:OP}.
 For a maximal chain $a_1\lessdot\dots\lessdot a_k$ in $P$, the corresponding  rows of the 
 augmented matrix $[\calA:c]$ are
\[
 \left[
 \begin{array}{rrrcrrrcc|c}
  1&0&0&\dots&0&0&0&\dots&0&0\\
  1&1&0&\dots&0&0&0&\dots&0&0\\
  0&1&1&\dots&0&0&0&\dots&0&0\\
  \vdots &\vdots &\vdots &\ddots&\vdots &\vdots&\vdots&\dots&\vdots&\vdots\\
  0&0&0&\dots&1&0&0&\dots&0&0\\
  0&0&0&\dots&1&1&0&\dots&0&0\\
  0&0&0&\dots&0&1&0&\dots&0&1
 \end{array}\,\right]\ .
\]
Observe that the sum of these rows is $[0:1]$.
Each row of $[\calA:c]$ has the form $(b,b\cdot F)$ (modulo 2), where $b$ is a
primitive normal vector to a facet $F$ of $\Delta$.

Since $P$ is not ranked modulo 2, it has two maximal chains of different parities.
Summing the rows of $[\calA:c]$ which correspond to facets given by the two chains 
gives a sum of an odd number of rows of $[\calA:c]$ which is equal to zero modulo 2.
\end{proof}

\subsection{Real solutions to systems of equations}

In~\cite{SS06} we considered systems,
 \begin{equation}\label{eq:system}
   f_1(x_1,\dotsc,x_n)\ =\ f_2(x_1,\dotsc,x_n)\ =\ \dotsb\ =\ 
   f_n(x_1,\dotsc,x_n)\ =\ 0\,,
 \end{equation}
where each $f_i$ is a real polynomial whose exponent vectors lie in $\Delta\cap\Z^n$, for
a fixed lattice polytope $\Delta$, called the \demph{Newton polytope} of the system.
When the exponent vectors $\Delta\cap\Z^n$ affinely span $\Z^n$, the solutions
to~\eqref{eq:system} correspond to a linear section $L\cap Y_\Delta$ of the real
projective toric variety $Y_\Delta$ corresponding to $\Delta$. 
Here $L\subset\R\P^\Delta$ is a linear subspace of codimension $n$.
Projecting from a general codimension one linear subspace $E$ of $L$, we may realize the
solutions to~\eqref{eq:system} as the fibers of a map
\[
   \pi_E\ \colon\ Y_\Delta\ \longrightarrow\ \R\P^n\,,
\]
to real projective space.
If $n$ is odd, then $\R\P^n$ is orientable.
If $Y_\Delta$ is also orientable, then fixing orientations, 
the map $\pi_E$ has a degree whose absolute value gives a lower bound on the cardinality
of a fiber of $\pi_E$, and thus on the number of real solutions 
to~\eqref{eq:system}. 

More generally, we may lift this projection to the spherical toric varieties
 \begin{equation}\label{Eq:projection_lift}
   \pi_E^+\ \colon\ Y^+_\Delta\ \longrightarrow\ \bS^n\,.
 \end{equation}
If $Y^+_\Delta$ is orientable, we fix an orientation and the absolute value of
the degree of $\pi_E^+$ is a lower bound on the number of solutions to the
system~\eqref{eq:system}. 
Changing orientations in each component if necessary, we may assume that the degree
  is divisible by the number of components of $(Y^+_\Delta)^\circ$.

This has the following consequence for lower bounds to 
systems of polynomial equations. 

\begin{theorem}\label{Th:final}
 Suppose that we have a system of polynomials~$\eqref{eq:system}$ with Newton polytope
 $\Delta$ where $\Delta\cap\Z^n$ affinely spans $\Z^n$ 
 whose solutions are a fiber of a projection map $\pi_E^+$~$\eqref{Eq:projection_lift}$.
 If there is a basis for $\{\pm1\}^{n+1}$ such that  $(-1)^{(b,b\cdot F)}$ is a
 product of an odd number of basis elements for every primitive normal vector $b$ to a
 facet $F$ of $\Delta$, then the absolute value of the degree of the map $\pi_E^+$ is a lower
 bound for the number of real solutions to~$\eqref{eq:system}$, and this lower bound is a
 multiple of the number of components of $(Y^+_\Delta)^\circ$.

 Moreover, the map $\pi_E^+$ does not have a degree if this condition is not satisfied.
\end{theorem}

\begin{remark}
 We did not need to consider the parity of $n$, for the condition of
 Theorem~\ref{T:orient_toric} implies that of Theorem~\ref{Th:spherical}.
 (A vector lies in a ray of the normal fan $\Sigma$ to $\Delta$ if and only if it is normal to
   a facet $F$ of $\Delta$.) 
\end{remark}

\subsection{Conclusions}
We characterized the orientability of $Y_\Delta$ and $Y^+_\Delta$,
  which implies that the corresponding polynomial system has lower bounds on its number of
  real solutions, expressed as 
the degree of a projection $\pi_E$ or $\pi_E^+$. 
These degrees have been computed for polynomial systems from posets~\cite{SS06} and those
from foldable triangulations~\cite{JW07,JZ12,SS06}.
Our  characterization
of orientability replaces the condition in \cite{SS06} that a variety is Cox-oriented
and therefore strengthens  the results of \cite{SS06}, particularly Theorem~3.5. 

\providecommand{\bysame}{\leavevmode\hbox to3em{\hrulefill}\thinspace}
\providecommand{\MR}{\relax\ifhmode\unskip\space\fi MR }
\providecommand{\MRhref}[2]{%
  \href{http://www.ams.org/mathscinet-getitem?mr=#1}{#2}
}
\providecommand{\href}[2]{#2}

\end{document}